\documentclass[12pt]{amsart}
\usepackage{amssymb}
\usepackage[mathscr]{eucal}
\usepackage{epsf}
\usepackage{epsfig}
\usepackage{color}
\DeclareGraphicsExtensions{.pstex,.eps,.epsi,.ps}

 \newtheorem{thm}{Theorem}[section]
 \newtheorem{cor}[thm]{Corollary}
 \newtheorem{lem}[thm]{Lemma}
 
 \theoremstyle{definition}
 
 \theoremstyle{remark}
 
 \numberwithin{equation}{section}

 \newcommand{\tr}{\textbf{tr}}

 \newcommand{\ric}{\textbf{Rc}}

 \newcommand{\Rm}{\textbf{Rm}}

\newcommand{\e}{\epsilon}

\begin{document}

\title[Warped Product Splitting Through Weak KAM]{A Warped Product Splitting Theorem Through Weak KAM Theory}

\author{Paul W.Y. Lee}
\email{wylee@math.cuhk.edu.hk}
\address{Room 216, Lady Shaw Building, The Chinese University of Hong Kong, Shatin, Hong Kong}

\date{\today}

\begin{abstract}
In this paper, we strengthen the splitting theorem proved in \cite{LiWa,LiWa2} and provide a different approach using ideas from the weak KAM theory.
\end{abstract}

\maketitle


\section{Introduction}

The classical splitting theorem of Cheeger-Gromoll states that if a manifold with non-negative Ricci curvature contains a line, then it is isometric to a product of a submanifold with the real line equipped with the product metric. Here a line means a geodesic which is the image of the real line and each segment of this geodesic is minimizing between its end-points.

In \cite{LiWa,LiWa2} (extending earlier works in \cite{WiYa,CaGa,Wa}), a version of the theorem for negative Ricci curvature is proved for manifolds of dimension greater than two (and we assume this condition for the rest of the paper). Here is the statement.
\begin{thm}
Let $M$ be a Riemannian manifold with the Ricci curvature bounded below by $-\frac{\lambda_1(n-1)}{n-2}$, where $\lambda_1$ is the first eigenvalue of the Laplace-Beltrami operator. Then one of the following two holds
\begin{enumerate}
\item $M$ has at most one infinite volume end,
\item $M$ is isometric to the warped product $\mathbb{R}\times N$ with metric given by $$dt^2+\cosh^2\left(t\sqrt{\frac{\lambda}{n-2}}\right)g_N,$$ where $N$ is a compact submanifold of $M$ and $g_N$ is the Riemannian metric of $N$ induced by that of $M$.
\end{enumerate}
\end{thm}

The above theorem says nothing about the situation when $M$ has one infinite and one finite volume end. In fact, the manifold $\mathbb{R}\times N$ equipped with the warped product metric $dt^2+\exp(2t)g_N$ is such an example. This example appeared in the same paper \cite{LiWa}. The followings is the main theorem in \cite{LiWa2} which deals with finite volume ends though with a different and stronger condition.

\begin{thm}
Let $M$ be a Riemannian manifold with the Ricci curvature bounded below by $-(n-1)$ and $\lambda_1\geq\frac{(n-1)^2}{4}$. Then one of the following two holds
\begin{enumerate}
\item $M$ has at most one end,
\item $M$ is isometric to the warped product $\mathbb{R}\times N$ with metric given by $$dt^2+\exp\left(2t\right)g_N,$$ where $N$ is a compact submanifold of $M$ and $g_N$ is the Riemannian metric of $N$ induced by that of $M$.
\end{enumerate}
\end{thm}

In this paper, we prove the following theorems strengthening the above results. Below the function $L:TM\to \mathbb{R}$ is the Lagrangian defined by
\[
L(x,v) = \frac{1}{2}|v|^2+g^{\frac{2n-2}{n-2}}(x).
\]

\begin{thm}\label{main0}
Let $g$ be a positive function which satisfies $\Delta g\leq -\lambda g$ for some positive $\lambda$. Assume that
\begin{enumerate}
\item the Ricci curvature bounded below by $-\frac{\lambda(n-1)}{n-2}$,
\item there is a curve $\gamma:\mathbb{R}\to M$ such that, for all real numbers $a<b$,
$\gamma\Big|_{[a,b]}$ is a minimizer of the following minimization problem
\[
\inf_{\sigma(a)=\gamma(a),\sigma(b)=\gamma(b)}\int_{a}^b L(\sigma(t),\dot\sigma(t))\,dt,
\]
\item $\int_\mathbb{R}L(\gamma(t),\dot\gamma(t))dt$ is finite.
\end{enumerate}
Then $M$ is isometric to the manifold $\mathbb{R}\times N$ equipped with the warped product metric
$$dt^2+\cosh^2\left(t\sqrt{\frac{\lambda}{n-2}}\right)g_N,$$
where $N$ is a submanifold of $M$, and $g_N$ is the Riemannian metric induced by that of $M$.
\end{thm}

\begin{thm}\label{main}
Let $g$ be a positive function which satisfies $\Delta g\leq -\lambda g$ for some positive $\lambda$. Assume that
\begin{enumerate}
\item the Ricci curvature bounded below by $-\frac{\lambda(n-1)}{n-2}$,
\item the infimum $\inf_{\gamma(0)=x}\int_{-\infty}^T
L(\gamma(t),\dot\gamma(t))\,dt$ is finite,
\item there is a sequence of minimizers $\gamma_i:(-\infty,T_i]\to M$ of the above problem with $T$ replaced by $T_i$ such that $\gamma_i(0)$ stays bounded as $i\to \infty$.
\end{enumerate}
Then $M$ is isometric to the manifold $\mathbb{R}\times N$ equipped with the warped product metric
\[
dt^2+\exp\left(2t\sqrt{\frac{\lambda}{n-2}}\right)g_N,
\]
where $N$ is a submanifold of $M$, and $g_N$ is the Riemannian metric induced by that of $M$.
\end{thm}

Note that the eigenvalue $\lambda$ in Theorem \ref{main} is not the first eigenvalue. The function $g$ is given by $\exp(-(n-2)t)$ with eigenvalue $n-2$. However, the first eigenvalue is given by $\frac{(n-1)^2}{4}$ which achieves the lower bound obtained in \cite{Ch} (see \cite{LiWa,LiWa2} for more detail).

The second assumption of Theorem \ref{main} means that the eigenfunction $g$ has certain decay in one direction. The last assumption of Theorem \ref{main} is an analogue of the existence of a line in the Cheeger-Gromoll splitting theorem. The condition is slightly more complicated since minimizers no longer have unit speed and minimizer which is the image of the real line might not exist since minimizers could go to infinity in finite time. In fact, this is the case for the models. Note also that there is no condition on ends of the manifold $M$. In particular one end is allowed and $N$ is not necessarily compact.

The proof of Theorem \ref{main} uses tools from the weak KAM theory. This theory started from the paper \cite{LPV} and later becomes a major tool in understanding Hamiltonian systems (see \cite{Fa} and references therein). The whole theory begins with various techniques in obtaining viscosity solutions, called weak KAM solutions, to Hamilton-Jacobi equations of the form $H(x,\nabla f(x))=H_0$ for some constant $H_0$, where $H$ is a Hamiltonian defined on the cotangent bundle of a manifold. The case when $M$ is a torus is dealt with in \cite{LPV}. It is later extended to all compact manifolds (see \cite{Fa2,Fa}). These results assume that the Hamiltonian $H$ is uniformly convex in the directions of the fibres. There are also more recent development in the case when this convexity assumption is dropped (see \cite{Ev2}). In this paper, we deal with the case when  $M$ is non-compact (a closely related result, which is not applicable in our situation, is contained in \cite{FaEz}) and that the Hamiltonian $H$ is of mechanical type with kinetic energy defined by the Riemannian metric and the potential defined by a power of the eigenfunction. In this case, many solutions are possible and the constant $H_0$ is not unique contrary to the compact case. The main work is devoted to the existence and regularity of an appropriate weak KAM solution in this situation.

The organization of this paper is as follows. In Section \ref{SecMech}, we establish the existence of the weak KAM solution mentioned above. In Section \ref{SecRiccati} and \ref{SecHar}, we show that the weak KAM solution is super-harmonic under the first and the second assumptions of Theorem \ref{main}. In Section \ref{SecRegular}, the last assumption of Theorem \ref{main} is used to show that the weak KAM solution is harmonic. Finally, in Section \ref{SecRigid}, we show that the gradient flow of the weak KAM solution provide the isometry stated in Theorem \ref{main}.

\section*{Acknowledgements}
I would like to thank those who supported me throughout my career.

\smallskip

\section{Mechanical Hamiltonians and Weak KAM}\label{SecMech}

In this section, we introduce the mechanical Hamiltonians and the corresponding results needed in this paper.

Let $\left<\cdot,\cdot\right>$ be a Riemannian metric defined on a manifold $M$ and let $V:M\to\mathbb{R}$ be a smooth function. A function $L:TM\to\mathbb{R}$ on the tangent bundle $TM$ of $M$ is a mechanical Lagrangian if it is of the form $L(x,v)=\frac{1}{2}|v|^2+V(x)$. The Riemannian metric also defines an identification of $TM$ with the cotangent bundle $T^*M$ by $v\mapsto \left<v,\cdot\right>$. The induces a Riemannian metric on $T^*M$ and it is denoted by the same symbol. The mechanical Hamiltonian $H:T^*M\to\mathbb{R}$ corresponding to $L$ is given by $H(x,p)=\frac{1}{2}|p|^2-V(x)$. Below, we consider $H$ as a function on the tangent bundle by  identifying tangent and cotangent bundle via the Riemannian metric.

Let $f:M\to\mathbb{R}$ be a bounded continuous function and let $S_Tf$ be the Lax-Oleinik semigroup of $f$  defined by
\begin{equation}\label{LO}
S_T f(x)=\inf_{\gamma(T)=x}\left(f(\gamma(0))+\int_0^TL(\gamma(s),\dot\gamma(s))ds\right)
\end{equation}
where the infimum is taken over absolutely continuous curves $\gamma:[0,T]\to M$ such that $\gamma(T)=x$.

Assuming that $M$ is compact, the weak KAM theorem states that there is a unique constant $c$ and a function $f$ such that $S_tf=f-ct$ (in fact this holds for more general Tonelli Lagrangian, see \cite{Fa}). It also follows that $f$ is a viscosity solution to the Hamilton-Jacobi equation $H(x,df_x)=c$ (see \cite{Ev} and references therein for more details on viscosity solution).

In this paper, the manifold $M$ is non-compact. In this case, the constant $c$ is no longer unique and we are interested in the case when $c=0$. More precisely, we have the following weak KAM type result.

\begin{thm}\label{wkam}
Assume $V>0$ and there exists a curve $\sigma:(-\infty,0]\to M$ such that $\int_{-\infty}^0L(\sigma(t),\dot\sigma(t))dt<\infty$. Then
\begin{enumerate}
    \item a subsequence of $S_Tf$ with $f=0$ converges locally uniform to a locally semi-concave function $F$,
    \item $F$ is given by
    \[
F(x)=\inf_{\gamma:(-\infty,0]\to M, \gamma(0)=x}\int_{-\infty}^0L(\gamma(t),\dot\gamma(t))dt,
\]
\item there is a minimizer $\gamma$ to the above problem and it satisfies $\frac{D^2}{dt^2}\gamma=\nabla V(\gamma)$ and $\dot\gamma(0)$ is in $\nabla^* F(x)$, where $\nabla^*F(x)$ denotes the reachable gradient of $F$ at $x$ (i.e. $v$ is in $\nabla^*F(x)$ if there is a sequence of points $x_i$ such that $\nabla F(x_i)$ exists and its limit is $v$),
\item $F$ is a viscosity solution to the Hamilton-Jacobi equation $$H(x,\nabla F(x))=0$$ (i.e. $H(x,p_+)\leq 0$ for $p_+$ in the super-differential $\nabla^+F(x)$ of $F$ at $x$ and $H(x,p_-)\geq 0$ for $p_-$ in the sub-differential $\nabla^-F(x)$ of $F$ at $x$).
\end{enumerate}
\end{thm}

\begin{proof}
First, note that the property $L(x,-v)=L(x,v)$  is used throughout this paper without mentioning. By the first variational formula, minimizers $\gamma$ of the problem (\ref{LO}) satisfies
\begin{equation}\label{minimizer}
\frac{D^2}{ds^2}\gamma(s)=\nabla V(\gamma(s)),\quad \dot\gamma(0)=\nabla f(\gamma(0)),
\end{equation}
where $\frac{D}{ds}$ denotes the covariant derivative.

Moreover, $\nabla (S_Tf)(x)=\dot\gamma(T)$ if $S_Tf$ is differentiable at $x$. Note that $S_Tf$ is locally semi-concave (see \cite{Fa,Vi}). In particular, it is twice differentiable almost everywhere (see \cite{EvGa,Vi}).

Let $\gamma$ be a minimizer of (\ref{LO}) with $f\equiv 0$. Since the Hamiltonian is conserved along any minimzer, it follows that
\begin{equation}\label{Hconserved}
\frac{1}{2}|\dot\gamma(t)|^2=V(\gamma(t))-V(\gamma(0)).
\end{equation}
Therefore,
\[
\int_0^TL(\gamma(s),\dot\gamma(s))ds \geq TV(\gamma(0)).
\]
So if $S_Tf(x)$ stays bounded independent of $T$, then $\gamma(0)\to\infty$ and $V(\gamma(0))\to 0$ as $T\to\infty$.

Since $V\geq 0$, $S_Tf$ is non-negative. By assumption,
\[
S_Tf(x)=\int_0^TL\left(\sigma(s-T),\frac{d}{ds}\sigma(s-T)\right)ds=\int_{-T}^0L\left(\sigma(s),\dot\sigma(s)\right)ds
\]
is bounded above independent of $T$.

On the other hand, by (\ref{Hconserved}), $|\nabla S_Tf(x)|=|\dot\gamma(T)|\leq V(x)$ wherever $S_Tf$ is differentiable. Therefore, $S_Tf$ is locally Lipschitz and the constant is independent of $T$ on any compact subset of $M$ (each $S_Tf$ is locally semi-concave see \cite{Fa,Vi}). Therefore, a subsequence of $T\mapsto S_T f$ has a limit $F$ which is locally Lipschitz.

Next, we show that $F$ is given by
\begin{equation}\label{F}
F(x)=\inf_{\gamma(0)=x}\int_{-\infty}^0L(\gamma(t),\dot\gamma(t))dt
\end{equation}
and $S_TF=F$. In particular, $F$ is locally semi-concave.

For the proof of this, let us assume first that $F$ is differentiable at $x$. Assume that there is a sequence of time $T_i$ such that $S_{T_i}f$ converges to $F$ as $i\to \infty$. Let $\gamma_i$ be a minimizer of (\ref{LO}) with $T$ replaced by $T_i$. Let $\bar\gamma_i(t)=\gamma_i(t+T_i)$. The function $S_{T_i}F$ satisfies
\[
S_{T_i}F(x)=\int_{-T_i}^0L(\bar\gamma_i(t),\dot{\bar\gamma}_i(t))dt\geq\int_{-T}^0\frac{1}{2}|\dot{\bar\gamma}_i(t)|^2dt\geq\frac{d^2(x,\bar\gamma_i(-T))}{2T}.
\]
for all large enough $i$.

It also follows from (\ref{Hconserved}) that
\[
\frac{1}{2}|\dot\gamma_i(t)|^2\leq V(\gamma_i(t)).
\]
It is bounded above by a constant independent of $i$ on compact subsets. By differentiating (\ref{minimizer}), it also follows that the higher derivatives of $\gamma_i$ are locally bounded independent of $i$. Therefore, $\gamma_i$ converges in $C^k$ norms locally uniformly to a curve $\bar\gamma:(-\infty,0]\to M$ such that $\gamma(0)=x$. Moreover, $\int_{-\infty}^0L(\bar\gamma(t),\dot{\bar\gamma}(t))dt$ is bounded. It follows that
\[
\begin{split}
&F(x)=\lim_{i\to\infty}S_{T_i}F(x)\\
&=\lim_{i\to\infty}\int_{-T_i}^0L(\bar\gamma_i(t),\dot{\bar\gamma}_i(t))dt \\
&= \int_{-\infty}^0L(\bar\gamma(t),\dot{\bar\gamma}(t))dt.
\end{split}
\]

On the other hand, let $\tilde\gamma_i$ be a minimizing sequence for the right hand side of (\ref{F}).
\[
\inf_{\tilde\gamma(0)=x}\int_{-\infty}^0L(\tilde\gamma_i(t),\dot{\tilde\gamma}_i(t))dt=\lim_{i\to\infty}\int_{-\infty}^0L(\tilde\gamma_i(t),\dot{\tilde\gamma}_i(t))dt\geq S_{T_k}F(x)
\]
for all positive integer $k$. Therefore, (\ref{F}) holds and $\bar\gamma$ is a minimizer. It also follows that
\[
\begin{split}
&S_TF(x)=\inf_{\gamma(T)=x}\left(F(\gamma(0))+\int_0^TL(\gamma(s),\dot\gamma(s))ds\right)\\
&=\inf_{\gamma(T)=x,\tilde\gamma(0)=\gamma(0)}\left(\int_{-\infty}^0 L(\tilde\gamma(s),\dot{\tilde\gamma}(s))ds+\int_0^TL(\gamma(s),\dot\gamma(s))ds\right)\\
&=\inf_{\gamma(T)=x}\left(\int_{-\infty}^T L(\gamma(s),\dot{\gamma}(s))ds\right)=F(x).
\end{split}
\]

Let $x_i$ be a sequence of points at which $F$ is differentiable and let $\gamma_i$ be a minimizer of the right hand side of (\ref{F}) with $x$ replaced by $x_i$. The same argument as above shows that $\gamma_i\Big|_{[-T,0]}$ is uniformly bounded. An argument using the formula of first variation shows that $\dot\gamma_i(0)=\nabla F(x_i)$ which is bounded since $F$ is locally Lipschitz. By conservation of Hamiltonian and (\ref{minimizer}), it also follows that the higher derivatives of $\gamma_i\Big|_{[-T,0]}$ are also bounded independent of $i$. Therefore, $\gamma_i$ converges in $C^k$ norm locally uniformly to a curve $\gamma:(-\infty,0]\to M$. By continuity of $F$, $\gamma$ is a minimizer of the right hand side of (\ref{F}). It follows that $\dot\gamma(0)$ is contained in the reachable gradient $\nabla^*F(x)$ of $F$ at $x$. Recall that $v$ is in $\nabla^*F(x)$ if there is a sequence of points $x_i$ converging to $x$ such that $\nabla F(x_i)$ exists and $\lim_{i\to\infty}\nabla F(x_i)=v$.

Since $S_tF(x)=F(x)$, there is a curve $\sigma:(0,\infty]\to M$ such that the followings hold for all $t>0$
\[
F(\sigma(t))=S_tF(\sigma(t))=F(\sigma(0))+\int_{0}^tL(\sigma(s),\dot\sigma(s))ds
\]
and $\dot\sigma(0)=\nabla F(\sigma(0))$.

If $F$ is differentiable at $x=\sigma(0)$, then the following holds by letting $t\to 0$ in the above equation
\begin{equation}\label{HJ}
H(x,\nabla F(x))=0.
\end{equation}

By continuity, the same holds with $\nabla F(x)$ replaced by the reachable gradient. Since $F$ is semi-concave and it is super-differentiable everywhere. Finally, if $p$ is in the super-differential of $F$ at $x$, then it is a convex combination of elements $p_0$ and $p_1$ in the reachable gradient of $F$ at $x$ (see \cite{CaSi}). It follows that
\[
H(x,p)=H(x,(1-t)p_0+tp_1)\leq (1-t)H(x,p_0)+tH(x,p_1)=0.
\]
\end{proof}

\smallskip

\section{Matrix Riccati Equations}\label{SecRiccati}

In this section, we consider the Hessian of the function $F$ along a minimizer. It satisfies a matrix Riccati equation and this is the key in obtaining control on the second derivative.

Let $t\mapsto \gamma(t,s)$ be a family of curves parametrized by $s$ which satisfy
\[
\frac{D^2}{dt^2}\gamma(t,s)=\nabla V(\gamma(t,s)).
\]

Let $X(t)=\frac{d}{ds}\gamma(t,s)\Big|_{s=0}$. It satisfies
\begin{equation}\label{var}
\begin{split}
&\frac{D^2}{dt^2}X(t)+\Rm\left(X(t),\dot\gamma(t,0)\right)\dot\gamma(t,0)\\
& =\frac{D}{ds}\frac{D^2}{dt^2}\gamma(t,s)\Big|_{s=0}=\nabla^2 V\left(X(t)\right).
\end{split}
\end{equation}
It follows that $X(t)$ is completely determined by its initial conditions $X(0)$ and $\frac{D}{dt}X(t)\Big|_{t=0}$.

Let $v_1(t),...,v_n(t)$ be a family of orthonormal frames defined along $\gamma(t,0)$ such that $v_1(t)=\frac{1}{|\dot\gamma(t)|}\dot\gamma(t)$ and $\frac{D}{dt}v_i(t)$ is contained in $\mathbb{R} v_1(t)$, $i=2,...,n$. Existence of this family will be established below. Let $E(t)=(v_1(t),...,v_n(t))^T$. It satisfies $\dot E(t)=A(t)E(t)$, where
\[
A(t)=\left(
\begin{array}{cc}
0 & A_2(t)\\
-A_2(t)^T & 0
\end{array}\right)
\]
and $A_2(t)=\left(
\begin{array}{ccc}
\frac{\left<\nabla V(\gamma(t)),v_2(t)\right>}{|\dot\gamma(t)|} & \cdots & \frac{\left<\nabla V(\gamma(t)),v_n(t)\right>}{|\dot\gamma(t)|}
\end{array}\right)$.

Let $B(t)$ be the matrix such that $X(t)=\sum_{j=1}^nB_{ij}(t)v_j(t)$ is a solution of (\ref{var}) with initial condition $X(0)=v_i(0)$ and $\frac{D}{dt}X(t)\Big|_{t=0}=\sum_{j=1}^n\left(\dot B_{ij}(0)+A_{ij}(0)\right)v_j(0)$.
Finally let $S(t)=B(t)^{-1}\dot B(t)+A(t)$.

\begin{lem}\label{griccati}
Assume that $H(\gamma(t,s),\dot\gamma(t,s))=0$ and the Ricci curvature $\ric$ of the manifold is bounded below by a constant $K$. Then $s(t):=\tr(S(t))$ satisfies
\[
\begin{split}
&\dot s(t)+\frac{s(t)^2}{n-1}-\frac{2s(t)}{n-1}\frac{\left<\nabla V(\gamma(t)),\dot\gamma(t)\right>}{|\dot\gamma(t)|^2}+\frac{n\left<\nabla V(\gamma(t)),\dot\gamma(t)\right>^2}{(n-1)|\dot\gamma(t)|^4}\\
&+\frac{2|P\nabla V|^2_{\gamma(t)}}{|\dot\gamma(t)|^2}+2KV(\gamma(t))-\Delta V(\gamma(t))\leq 0.
\end{split}
\]
\end{lem}

If we specialize to the case $V=\frac{1}{2}g^{\frac{2n-2}{n-2}}$, where $\Delta g\leq-\lambda g$, then the above lemma implies that

\begin{cor}\label{riccati}
Assume that the Ricci curvature $\ric$ of the manifold is bounded below by a constant $-\frac{\lambda(n-1)}{n-2}$. Let $V=\frac{1}{2}g^{\frac{2n-2}{n-2}}$ and $\Delta g\leq-\lambda g$. Then $s(t)$ satisfies
\[
\begin{split}
&\dot s(t)+\frac{s(t)^2}{n-1}-\frac{2s(t)}{n-2}\frac{\left<\nabla g(\gamma(t)),\dot\gamma(t)\right>}{g(\gamma(t))}\leq 0.
\end{split}
\]
\end{cor}

\begin{proof}[Proof of Lemma \ref{griccati}]
First, let us establish the existence of the family $v_1(t),...,v_n(t)$. Let $v_1(0),...,v_n(0)$ be an orthonormal frame defined at $y=\gamma(0,0)$ such that $v_1(0)=\frac{1}{|\dot\gamma(0)|}\dot\gamma(0)$. Let $v_1(t),...,v_n(t)$ be family of frames defined along $\gamma(t,0)$ such that $v_1(t)=\frac{1}{|\dot\gamma(t,0)|}\dot\gamma(t,0)$ and $\dot v_i(t)$ is contained in the span of $v_1(t)$ for all $i=2,...,n$.
Indeed, suppose $v_1(t),w_2(t),...,w_n(t)$ is an orthonormal frame. Let $v_i(t)=\sum_{j=1}^nQ_{ij}(t)w_j(t)$. It follows that
\[
\frac{D}{dt}v_i(t)=\sum_{j=1}^n\dot Q_{ij}(t)w_j(t)+\sum_{j=1}^nQ_{ij}(t)\frac{D}{dt}w_j(t)
\]
and so
\[
\left<\frac{D}{dt}v_i(t),v_k(t)\right>=\sum_{j=1}^n\dot Q_{ij}(t)Q_{kj}(t)+\sum_{j,l=1}^nQ_{ij}(t)Q_{kl}\left<\frac{D}{dt}w_j(t),w_l(t)\right>
\]

The matrix $D(t)$ with $jk$-th entry equal to $\left<\frac{D}{dt}w_j(t),w_k(t)\right>$ is skew-symmetric. Therefore, there is a solution to the equation $\dot Q(t)=Q(t)D(t)$. It follows that
$\left<\frac{D}{dt}v_i(t),v_k(t)\right>=0$ for all $i,j\neq 1$. Hence, $\dot v_i(t)$ is in the span of $v_1(t)$ as claimed.

A computation shows that
\[
\begin{split}
&\frac{D}{dt}v_1(t)=\frac{1}{|\dot\gamma(t)|}\frac{D^2}{dt^2}\gamma(t)-\frac{\left<\frac{D^2}{dt^2}\gamma(t),\dot\gamma(t)\right>}{|\dot\gamma(t)|^3}\dot\gamma(t)\\
&=\frac{1}{|\dot\gamma(t)|}\sum_{i=2}^n\left<\nabla V(\gamma(t)),v_i(t)\right>v_i(t).
\end{split}
\]
Let $E(t)=(v_1(t),...,v_n(t))^T$. Then $\dot E(t)=A(t)E(t)$, where
\[
A(t)=\left(\begin{array}{cc}
0 & A_2(t)\\
-A_2(t)^T & 0
\end{array}\right)
\]
and $A_2(t)=\left(\begin{array}{ccc}
\frac{\left<\nabla V_{\varphi_t},v_2(t)\right>}{|\dot\varphi_t|} & \cdots & \frac{\left<\nabla V_{\varphi_t},v_n(t)\right>}{|\dot\varphi_t|}
\end{array}\right)$.

By (\ref{var}), the family $B(t)$ satisfies
\begin{equation}\label{B}
\begin{split}
&\ddot B(t)+2\dot B(t)A(t)+B(t)\dot A(t)\\
& +B(t)A(t)^2+B(t)R(t)-B(t)W(t)=0,
\end{split}
\end{equation}
where $W_{ij}(t)=\nabla^2V(v_i(t),v_j(t))$.

Recall that $S(t)=B(t)^{-1}\dot B(t)+A(t)$. It satisfies
\begin{equation}\label{S}
\begin{split}
&\dot S(t)=-B(t)^{-1}\dot B(t)B(t)^{-1}\dot B(t)+B(t)^{-1}\ddot B(t)+\dot A(t)\\
&=-B(t)^{-1}\dot B(t)B(t)^{-1}\dot B(t)\\
&-2B(t)^{-1}\dot B(t)A(t)-A(t)^2-R(t)+W(t)\\
&=-S(t)^2-S(t)A(t)-A(t)^TS(t)-R(t)+W(t).
\end{split}
\end{equation}

Let us split the matrices
\[
S(t)=\left(
\begin{array}{cc}
S_1(t) & S_2(t)\\
S_2(t)^T & S_3(t)\end{array}\right),
\]
\[
W(t)=\left(
\begin{array}{cc}
W_1(t) & W_2(t)\\
W_2(t)^T & W_3(t)\end{array}\right),
\]
and
\[
R(t)=\left(
\begin{array}{cc}
0 & 0\\
0 & R_3(t)\end{array}\right).
\]
Here $S_3(t)$ and $W_3(t)$, and $R_3(t)$ are $(n-1)\times (n-1)$ blocks.

The block $S_3(t)$ satisfies
\[
\begin{split}
&0=\dot S_3(t)+S_2(t)^TS_2(t)+S_3(t)^2\\
&+S_2(t)^TA_2(t)+A_2(t)^TS_2(t)+R_3(t)-W_3(t).
\end{split}
\]

Its trace $s_3(t)=\tr(S_3(t))$ satisfies
\begin{equation}\label{S3}
\begin{split}
&0=\dot s_3(t)+|S_2(t)|^2+2\left<S_2(t),A_2(t)\right>\\
&+|S_3(t)|^2+\tr(R_3(t))-\tr(W_3(t))
\end{split}
\end{equation}

On the other hand, by differentiating the condition $H(\gamma(t,s),\dot\gamma(t,s))=0$, we also have
\[
|S_2(t)|^2=\frac{|P\nabla V(\gamma(t,0))|^2}{|\dot\gamma(t,0)|^2}=\left<S_2(t),A_2(t)\right>
\]
and
\[
S_1(t)=\frac{\left<\nabla V(\gamma(t)),\dot\gamma(t)\right>}{|\dot\gamma(t)|^2}.
\]

Therefore,
\begin{equation}\label{S1}
\begin{split}
&0=\dot S_1(t)+S_1(t)^2+|S_2(t)|^2-2\left<S_2(t),A_2(t)\right>-W_1(t)\\
&=\dot S_1(t)+\frac{\left<\nabla V,\nabla F\right>^2_{\varphi_t(y)}}{|\nabla F|^4_{\varphi_t(y)}}-\frac{|P\nabla V|^2_{\varphi_t(y)}}{|\nabla F|^2_{\varphi_t(y)}}-W_1(t).
\end{split}
\end{equation}

Let $s(t)=S_1(t)+s_3(t)$. By combining (\ref{S1}) and (\ref{S3}), we obtain
\[
\begin{split}
&\dot s(t)+|S_3(t)|^2+\frac{\left<\nabla V(\gamma(t)),\dot\gamma(t)\right>^2}{|\dot\gamma(t)|^4}\\
&+\frac{2|P\nabla V(\gamma(t))|^2}{|\dot\gamma(t)|^2}+\tr(R_3(t))-\Delta V(\gamma(t))=0.
\end{split}
\]
The result follows from this by applying the Ricci curvature lower bound.
\end{proof}

\begin{proof}[Proof of Corollary \ref{riccati}]
Let $V=\frac{1}{2}g^{\frac{2n-2}{n-2}}$ and $\Delta g=-\lambda g$. A computation shows that
\[
\begin{split}
&\nabla V=\frac{n-1}{n-2}g^{\frac{n}{n-2}}\nabla g\\
&\nabla^2 V(v_1(t),v_1(t))=\frac{n-1}{n-2}\frac{n}{n-2}g^{\frac{2}{n-2}}\left<\nabla g,v_1(t)\right>^2\\
&+\frac{n-1}{n-2}g^{\frac{n}{n-2}}\nabla^2 g(v_1(t),v_1(t))\\
&\Delta V=\frac{n-1}{n-2}\left(\frac{n}{n-2}g^{\frac{2}{n-2}}|\nabla g|^2-\lambda g^{\frac{2n-2}{n-2}}\right).
\end{split}
\]

By specializing the lemma to the case $K=-\frac{\lambda(n-1)}{n-2}$, we have
\[
\begin{split}
&\dot s(t)+\frac{s(t)^2}{n-1}-\frac{2s(t)}{n-1}\frac{\left<\nabla V(\gamma(t)),\dot\gamma(t)\right>}{|\dot\gamma(t)|^2}\leq 0.
\end{split}
\]
\end{proof}

\smallskip

\section{Super-Harmonicity of the Weak KAM Solution}\label{SecHar}

In this section, we show that the weak KAM solution found above is super-harmonic in the weak sense under the assumptions of Theorem \ref{main} (i.e. the distributional Laplacian acting on the weak KAM solution is non-positive).

Let $f$ be a locally semi-concave function such that
\begin{equation}\label{LOsg}
f(x)=\inf_{\gamma_0(T)=x}\left(f(\gamma_0(0))+\int_0^TL(\gamma_0(t),\dot\gamma_0(t))dt\right).
\end{equation}

\begin{thm}\label{superharmonic}
Under the assumptions in Corollary \ref{riccati}, the function $f$ is super-harmonic in the weak sense.
\end{thm}

For the proof, we will use the notations from the last section.

\begin{proof}
First, note that the distributional second derivative $D^2f$ of a semi-concave function is a measure \cite{EvGa}. By the Alexandrov theorem \cite{EvGa}, $f$ is twice differentiable almost everywhere and the absolute continuous part of $D^2f$ is given by this almost everywhere second derivative $\nabla^2f$. Moreover, the trace of the singular part of $D^2f$ is non-positive \cite{EvGa}. Therefore, it remains to show that the trace of $\nabla^2f$ is non-positive.

First, let us consider the Hamiltonian flow $\Psi_t:TM\to TM$ (again $T^*M$ and $TM$ are identified) defined by
\[
\begin{split}
&\frac{D^2}{dt^2}\pi(\Psi_t(x,v))=\nabla V(\pi(\Psi_t(x,v))),\\
&\frac{d}{dt}\pi(\Psi_t(x,v))=\Psi_t(x,v), \quad \Psi_0(x,v)=(x,v),
\end{split}
\]
where $\pi:TM\to M$ is the natural projection map.

This map is the flow of the Hamiltonian vector field and so it is $C^\infty$ in all variables. Let $x$ be a point where $f$ is twice differentiable and let $\varphi_t(y)=\Psi_t(y,-\nabla f(y))$. This map is differentiable in $x$ and the derivation of (\ref{S}) is still valid, where $S(0)$ is the matrix of $-\nabla^2f(x)$ with respect to an orthonormal frame. The arguments in Corollary \ref{riccati} is also valid.

Therefore, the conclusion of Corollary \ref{riccati} is valid and we have
\[
|\dot\gamma(t)|^{-1}\dot b(t)\leq-\frac{b(t)^2}{n-1} -\frac{(n-3)(n-1)b(t)}{(n-2)^2}\left<\nabla\log g(\gamma(t)),v_1(t)\right>
\]
where $b(t)=\frac{s(t)}{g(\gamma(t))^{\frac{n-1}{n-2}}}$, $\gamma(t)=\varphi_t(x)$, and $s(t)=\tr(S(t))$.

Next, we rescale time
\[
\frac{d}{dt} b(c(t))\leq-\frac{b(c(t))^2}{n-1} -\frac{(n-3)(n-1)b(c(t))}{(n-2)^2}\left<\nabla \log g(\gamma(c(t))),v_1(c(t))\right>,
\]
where $\dot c(t)=|\dot\gamma(c(t))|^{-1}$ and $c(0)=0$. This rescaling ensures that $\gamma(c(t))$ has constant speed and so does not go to infinity in finite time.

Let $\bar b(t)$ be the solution of the equation
\[
\frac{d}{dt}\bar b(t) =-\frac{\bar b(t)^2}{n-1} -\frac{(n-3)(n-1)\bar b(t)}{(n-2)^2}d(t),
\]
with initial condition $\bar b(0)=b(c(0))$, where $d(t)=\frac{d}{dt}\log g(\gamma(c(t)))$.

The function $\bar b$ can be found explicitly in terms of $g$ using the method in \cite{Lev}. For this, let $M(t)$ be the fundamental solution of the equation
\[
\dot x(t)=\left(
\begin{array}{cc}
 -k\,d(t) & 0\\
\frac{1}{n-1} & k\,d(t)
\end{array}\right)x(t)
\]
which is given by
\[
M(t)=\left(
\begin{array}{cc}
\frac{g(x)^k}{g(\gamma(c(t)))^k} & 0\\
\int_0^t\frac{g(x)^kg(\gamma(c(t)))^k}{(n-1)g(\gamma(c(s)))^{2k}}ds & \frac{g(\gamma(c(t)))^k}{g(x)^k}
\end{array}\right),
\]
where $k=\frac{(n-3)(n-1)}{2(n-2)^2}$.

It follows that
\[
\bar b(t)=\left(\frac{g(x)^{2k}}{g(\gamma(c(t)))^{2k}}b(0)\right)\left(b(0)\int_0^t\frac{g(x)^{2k}}{(n-1)g(\gamma(c(s)))^{2k}}ds+1\right)^{-1}.
\]

By comparison principle \cite{Ro}, $b(c(t))\leq\bar b(t)$. Since $\lim_{t\to\infty}g(\gamma(c(t)))=0$, $\bar b(t)\to -\infty$ in finite time if $b(0)$ is negative. This gives a contradiction. Therefore, $-\Delta f(x)=s(0)\geq 0$ as claimed.
\end{proof}

\smallskip

\section{Regularity of the Weak KAM Solution}\label{SecRegular}

In this section, we prove higher regularity for the function $F$. In fact, we show that $F$ is harmonic and so $C^\infty$. For this, we need to consider another minimization problem and another weak KAM solution
\begin{equation}\label{minusF}
S_T(-F)(x)=\inf_{\gamma(T)=x}\left(-F(\gamma(0))+\int_0^TL(\gamma(t),\dot\gamma(t))dt\right).
\end{equation}

In the proof of the following theorem, we show that a subsequence of $S_T(-F)$ converges to $-F$ under the last assumption of Theorem \ref{main}. Then, the arguments from the previous section shows that it is both super- and sub-harmonic.

\begin{thm}
    Assume that the assumptions of Theorem \ref{main} holds. Then $F$ is a harmonic function. In particular, it is $C^\infty$.
\end{thm}
\begin{proof}
The family of functions $S_T(-F)$ is bounded below by $-F$. Indeed, let $\gamma_0$ be a minimizer of the problem (\ref{minusF}) and let $\gamma_1:(-\infty,0]\to M$ be a minimizer of (\ref{F}). Finally let $\gamma_2:(-\infty,0]\to M$ be the path
\[
\gamma_2(t)=\begin{cases}
\gamma_1(t+T) & \mbox{if } t\in(-\infty,-T]\\
\gamma_0(-t) & \mbox{if } t\in(-T,0].
\end{cases}
\]

It follows that
\[
\begin{split}
&S_T(-F)(x)= -F(\gamma_0(0))+\int_0^TL(\gamma_0(t),\dot\gamma_0(t))dt\\
&\geq -\int_{-\infty}^0 L(\gamma_2(t),\dot\gamma_2(t))dt+\int_0^TL(\gamma_0(t),\dot\gamma_0(t))dt\\
&=-\int_{-\infty}^{-T} L(\gamma_1(t),\dot\gamma_1(t))dt\\
&-\int_{-T}^0 L(\gamma_0(-t),\dot\gamma_0(-t))dt+\int_0^TL(\gamma_0(t),\dot\gamma_0(t))dt\\
&=-\int_{-\infty}^{-T} L(\gamma_1(t),\dot\gamma_1(t))dt =-F(x).
\end{split}
\]

If $S_T(-F)$ is differentiable at $x$, then $\dot\gamma(T)=\nabla S_T(-F)(x)$, $-\dot\gamma(0)$ is in $\nabla^*F(\gamma(0))$, and $\frac{D^2}{dt^2}\gamma=\nabla V(\gamma)$. By the conservation of Hamiltonian and that $F$ is a viscosity solution of the Hamilton-Jacobi equation,
\[
\frac{1}{2}|\nabla S_T(-F)(x)|^2\leq \frac{1}{2}|\dot\gamma(0)|^2-V(\gamma(0))+V(x)=V(x).
\]

It follows that $S_T(-F)$ is locally Lipschitz with constant independent of $T$ on compact subsets. It follows that $S_T(-F)$ converges locally uniformly to a Lipschitz function $G$. It also follows that $F+G\geq 0$.

Next, we show that $T\mapsto S_T(-F)$ is monotone. It follows that the whole sequence converges to $G$. This, in turn, implies that $S_TG=G$. Monotonicity follows from $S_T(-F)\geq -F$ and the semi-group property $S_{t+h}(-F)=S_h(S_t(-F))$. Indeed,
\[
\begin{split}
&S_{t+h}f(x)=\inf_{\gamma(t+h)=x}\left(f(\gamma(0))+\int_0^{t+h} L(\gamma(s),\dot{\gamma}(s))ds\right)\\
&=\inf_{\gamma(t+h)=x}\left(f(\gamma(0))+\int_0^{t} L(\gamma(s),\dot{\gamma}(s))ds+\int_t^{t+h} L(\gamma(s),\dot{\gamma}(s))ds\right)\\
&=\inf_{\gamma(t+h)=x}\left(S_tf(\gamma(t))+\int_t^{t+h} L(\gamma(s),\dot{\gamma}(s))ds\right) =S_h(S_t(f))(x).
\end{split}
\]

Finally, assume that $S_tG$ is differentiable at $x$ and let $\gamma:[0,t]\to M$ be a minimizer of (\ref{LO}) with $\gamma(t)=x$ and $F$ replaced by $G$. It follows that
\[
\begin{split}
&S_tG(x)=G(\gamma(0))+\int_0^tL(\gamma(s),\dot\gamma(s))ds\\
&=\lim_{T\to\infty}\inf_{\bar\gamma(T)=\gamma(0)}\left(-F(\bar\gamma(0))+\int_0^TL(\bar\gamma(s),\dot{\bar\gamma}(s))ds+\int_0^tL(\gamma(s),\dot\gamma(s))ds\right)\\
&= G(x)
\end{split}
\]
as claimed.

Let $x$ be a point where $F$ is differentiable and let $\gamma:(-\infty,0]\to M$ be a minimizer of (\ref{F}). It follows that
\[
\begin{split}
&S_T(-F)(\gamma(-T))=-F(\gamma(0))+\int_0^TL\left(\gamma(-t),\frac{d}{dt}\gamma(-t)\right)dt\\
&=-\int_{-\infty}^{0}L(\gamma(t),\dot\gamma(t))dt+\int_{-T}^0L(\gamma(t),\dot\gamma(t))dt\\
&=-\int_{-\infty}^{-T}L(\gamma(t),\dot\gamma(t))dt=-F(\gamma(-T)).
\end{split}
\]
By applying the above arguments to a sequence of points on the minimizer given by the statement of the theorem, it follows that $G(x)+F(x)=0$ at a point $x$.

The argument in the previous section showed that $F+G$ is super-harmonic in the weak sense. Therefore, by strong maximum principle \cite{GiTr}, it must vanish. It follows that $F=-G$. This implies that $F$ is harmonic and so $C^\infty$.
\end{proof}

\smallskip

\section{Rigidity and the End of Proof of Theorem \ref{main}}\label{SecRigid}

In this section, we finish the proof of the theorem.  This is essentially the argument used in \cite{Wa,LiWa} written in the language used in this paper. In the following, we will use the same notations as the one in Section \ref{SecRiccati}.

The equality cases of (\ref{S1}) and (\ref{S3}) are satisfied. It follows that $S_3(t)$ is a multiple of identity, $P\nabla g(\varphi_t)=0$, and $\tr(R_3(t))=-\frac{\lambda (n-1)}{n-2}|\nabla F|_{\varphi_t}^2$. Therefore, $S_2=0$ and $A_2=0$. Since $\nabla F$ is $C^\infty$, its gradient flow provides a geodesic foliation for $M$. Moreover, the level sets of $g$ are orthogonal to $\nabla F$. A computation using the Hamilton-Jacobi equation shows that
\begin{equation}\label{ng}
\nabla g =\frac{\left<\nabla g,\nabla F\right>}{g^{\frac{2n-2}{n-2}}} \nabla F.
\end{equation}
and
\begin{equation}\label{nf}
\nabla F^2(v_1(t))=\frac{n-1}{n-2}g(\varphi_t)^{\frac{1}{n-2}}\nabla g(\varphi_t).
\end{equation}
It follows that
\[
S_1(t)=\frac{n-1}{n-2}g(\varphi_t)^{\frac{1}{n-2}}\left<\nabla g(\varphi_t),v_1(t)\right>.
\]

The function $F$ is harmonic. It follows that
\[
S_3(t)=-\frac{1}{n-2}g(\varphi_t)^{\frac{1}{n-2}}\left<\nabla g(\varphi_t),v_1(t)\right>I.
\]

The equation for $S_3$ becomes
\begin{equation}\label{newS3}
\begin{split}
&0=\dot a I+a^2I+R_3(t)-W_3(t).
\end{split}
\end{equation}

Since $\nabla g$ is a multiple of $\nabla F$, $W_3(t)$ is a multiple of identity and so is $R_3(t)$. Therefore, $R_3(t)=-\frac{\lambda}{n-2}|\nabla F|^2_{\varphi_t}I$ and

For all $i,j\neq 1$ and $i\neq j$,
\[
\begin{split}
&\nabla^2 g(v_i(t),v_j(t))=\frac{\left<\nabla g(\varphi_t),v_1(t)\right>}{g(\varphi_t)^{\frac{n-1}{n-2}}}\nabla^2 F(v_i(t),v_j(t)) \\
&=-\frac{\left<\nabla g(\varphi_t),v_1(t)\right>^2}{(n-2)g(\varphi_t)}\delta_{ij}
\end{split}
\]
by (\ref{ng}). Therefore,
\[
\nabla^2g(v_1(t),v_1(t))=-\lambda g(\varphi_t)+\frac{(n-1)\left<\nabla g(\varphi_t),v_1(t)\right>^2}{(n-2)g(\varphi_t)}.
\]

Next, let $\dot c(t)=g(\varphi_{c(t)})^{-\frac{n-1}{n-2}}$. The flow lines of $\varphi_{c(t)}$ are unit speed geodesics tangent to $\nabla F$. A computation using (\ref{nf}) shows that
\[
\begin{split}
&\frac{d}{dt}\log g(\varphi_{c(t)})=\left<\nabla \log g(\varphi_{c(t)}),v_1(c(t))\right>\\
&=\frac{\left<\nabla g(\varphi_{c(t)}),v_1(c(t))\right>}{g(\varphi_{c(t)})}
\end{split}
\]
and
\[
\begin{split}
&\frac{d^2}{dt^2}\log g(\varphi_{c(t)})\\
&=\frac{\left<\nabla^2 g(\varphi_{c(t)})v_1(t),v_1(t)\right>}{g(\varphi_{c(t)})}-\frac{\left<\nabla g(\varphi_{c(t)}),v_1(t)\right>^2}{g(\varphi_{c(t)})^2}\\
&=-\lambda +\frac{\left<\nabla g(\varphi_{c(t)}),v_1(t)\right>^2}{(n-2)g(\varphi_{c(t)})^2}\\
&=-\lambda +\frac{1}{n-2}\left(\frac{d}{dt}\log g(\varphi_{c(t)})\right)^2
\end{split}
\]

Using the method in \cite{Lev}, this can be solved explicitly. The fundamental solution $M(t)$ of
\[
\dot x=\left(
\begin{array}{cc}
0 & -\lambda\\
-\frac{1}{n-2} & 0
\end{array}\right)x.
\]
is given by
\[
M(t)=\left(
\begin{array}{cc}
\cosh\left(t\sqrt{\frac{\lambda}{n-2}}\right) & -\sqrt{\lambda(n-2)}\sinh\left(\sqrt{\frac{\lambda}{n-2}}t\right)\\
-\sqrt{\frac{1}{\lambda(n-2)}}\sinh\left(t\sqrt{\frac{\lambda}{n-2}}\right) & \cosh\left(t\sqrt{\frac{\lambda}{n-2}}\right)
\end{array}\right).
\]

It follows that
\[
\begin{split}
&\frac{d}{dt}\log g(\varphi_{c(t)})\\
&=\frac{d}{dt}\log\left(-\sqrt{\frac{1}{\lambda(n-2)}}\sinh\left(t\sqrt{\frac{\lambda}{n-2}}\right)b(0)+\cosh\left(t\sqrt{\frac{\lambda}{n-2}}\right)\right)^{-(n-2)}
\end{split}
\]
and so
\[
\begin{split}
&g(\varphi_{c(t)}(x))\\
&=g(x)\left(\cosh\left(t\sqrt{\frac{\lambda}{n-2}}\right)-\sqrt{\frac{1}{\lambda(n-2)}}\sinh\left(t\sqrt{\frac{\lambda}{n-2}}\right)|\nabla\log g|_x\right)^{-(n-2)}.
\end{split}
\]

Since $g$ is positive everywhere, $\lambda(n-2)\geq c^2$. Note that since $\nabla^2 g(\nabla F)=0$, $|\nabla g|$ is constant on level sets of $g$.

Finally recall that
\[
\dot B(t)=B(t)S(t),\quad B(0)=I,
\]
where $B(t)$ is the matrix representation of the derivative of the map $\varphi_t$.

Since
\[
S(t)=\left(
\begin{array}{cc}
\frac{n-1}{n-2}\left<\nabla(\log g),\nabla F\right>_{\varphi_t} & 0\\
0 & -\frac{1}{n-2}\left<\nabla(\log g),\nabla F\right>_{\varphi_t}I_{n-1},
\end{array}
\right)
\]
it follows that
\[
B(t)=\left(
\begin{array}{cc}
\frac{g(\varphi_t)^{\frac{n-1}{n-2}}}{g(x)^{\frac{n-1}{n-2}}} & 0\\
0 & \frac{g(x)^{\frac{1}{n-2}}}{g(\varphi_t)^{\frac{1}{n-2}}}I_{n-1}
\end{array}
\right).
\]
Therefore, $(t,y)\mapsto \varphi_{c(t)}(y)$ defines an isometry from $\mathbb{R}\times g^{-1}(g(x))$ equipped with the wrap product metric
\[
dt^2+\left(\cosh\left(t\sqrt{\frac{\lambda}{n-2}}\right)-\sqrt{\frac{c^2}{\lambda(n-2)}}\sinh\left(t\sqrt{\frac{\lambda}{n-2}}\right)\right)^2g_N
\]
to $M$.

The case $c=0$ does not satisfy the assumptions of the theorem. For other $c$'s, by a tranlation in time $t$, it reduces to the case $c=1$.

\smallskip

\section{Proof of Theorem \ref{main0}}

In this section, we give the proof of Theorem \ref{main0} which is very similar and simpler than the one of Theorem \ref{main}.

Let $\gamma:\mathbb{R}\to M$ be the path given in the statement of Theore \ref{main0}. Let $F_t:M\to\mathbb{R}$ be a family of functions defined by
\[
F_t(x)=\inf_{\sigma(0)=x,\sigma(t)=\gamma(-t)}\int_0^tL(\sigma(s),\dot\sigma(s))ds.
\]

By the assumptions, $F_t$ is bounded both above and below by a constant independent of $t$. Moreover, $F_t(x)$ is locally semi-concave (see \cite{Vi}) and $\nabla F_t(x)=-\dot\gamma_t(0)$ whenever $F_t$ is differentiable at $x$, where $\gamma_t$ is a minimizer of the above minimization problem.

By the conservation of the Hamiltonian,
\[
\frac{1}{2}|\dot\gamma_t(s)|^2-V(\gamma_t(s))=C_t
\]
for some constant $C_t$ depending only on $t$. Since
\[
\int_0^tL(\sigma(s),\dot\sigma(s))ds\geq |C_t|t
\]
is bounded above, it follows that $C_t\to 0$ as $t\to\infty$. In particular, the Lipschitz constants of $F_t$ can be chosen independent of $t$ on compact subsets. Therefore, a subsequence of $F_t$ converges to a locally Lipschitz function $F$. Moreover, by taking a further subsequence, we can also assume that $\gamma_t$ converges locally uniformly to a curve $\gamma:[0,\infty)\to M$ such that
\[
\frac{1}{2}|\dot\gamma|^2=V(\gamma).
\]

Let $x$ be a point where $F$ is differentiable and let $\gamma$ be the corresponding minimizer defined above. Let $\sigma_\epsilon$ be a compactly supported variation of $\gamma$ near the point $x$. Let $v=\frac{d}{d\e}\sigma_\e(0)\Big|_{\e=0}$.
\[
\begin{split}
&\left<\nabla F(x),v\right>\geq \frac{d}{d\e}\Big|_{\e=0}\int_0^\infty L(\sigma_\e(s),\dot\sigma_\e(s))ds\\
&=\int_0^\infty \left<\dot\sigma_\e(s),\frac{D}{d\e}\dot\sigma_\e(s)\right> +\left<\nabla V(\sigma_\e),\frac{D}{d\e}\sigma_\e(s)\right>ds\Big|_{\e=0}\\
&=\int_0^\infty \frac{d}{dt}\left<\dot\sigma_\e(s),\frac{D}{d\e}\Big|_{\e=0}\sigma_\e(s)\right>ds\\
&=-\left<\dot\gamma(0),v\right>
\end{split}
\]
The vector $v$ is arbitrary. It follows that $\nabla F(x)=-\dot\gamma(0)$ and so $\frac{1}{2}|\nabla F(x)|^2=V(x)$.

Next, we show that $S_tf=f$. For this, let $\gamma:[0,\infty)\to M$ be a curve such that $\gamma(0)=x$ and $f(x)=\int_0^\infty L(\gamma(s),\dot\gamma(s))ds$. It follows that
\[
\begin{split}
&f(x)=\int_t^\infty L(\gamma(s),\dot\gamma(s))ds+\int_0^t L(\gamma(s),\dot\gamma(s))ds\\
&=f(\gamma(t))+\int_0^t L(\gamma(s),\dot\gamma(s))ds\geq S_tf(x).
\end{split}
\]
For the reverse inequality, let $\gamma_1:[0,t]\to M$ be the minimizer of (\ref{LO}) with $T$ replaced by $t$. Let $\gamma_2:[0,\infty)\to M$ be a curve which satisfy $\gamma_2(0)=\gamma_1(0)$ and
\[
f(\gamma_1(0))=\int_0^\infty L(\gamma_2(t),\dot\gamma_2(t))dt.
\]
It follows that
\[
\begin{split}
&S_tf(x)=f(\gamma_1(0))+\int_0^tL(\gamma_1(s),\dot\gamma_1(s))ds\\
&=\int_0^\infty L(\gamma_2(t),\dot\gamma_2(t))dt+\int_0^tL(\gamma_1(s),\dot\gamma_1(s))ds.
\end{split}
\]
Recall that $\gamma_2$ is defined as a limit of a family of minimizers connecting $\gamma_1(0)$ and $\gamma(-t)$. It follows that
\[
\begin{split}
&S_tf(x)=f(\gamma_1(0))+\int_0^tL(\gamma_1(s),\dot\gamma_1(s))ds\\
&=\int_0^\infty L(\gamma_2(t),\dot\gamma_2(t))dt+\int_0^tL(\gamma_1(s),\dot\gamma_1(s))ds \geq f(x).
\end{split}
\]

It also follows that $F$ is a viscosity solution of the above Hamilton-Jacobi equation. A similar construction with $\gamma(-t)$ replaced by $\gamma(t)$ yields another solution $G$. By Theorem \ref{superharmonic}, both $F$ and $G$ are super-harmonic. Moreover,
\[
F_t(y)+G_t(y)-\int_{-t}^tL(\gamma(s),\dot\gamma(s))ds\geq 0
\]
since $\gamma$ is a minimizer between $\gamma(-t)$ and $\gamma(t)$. Therefore,
\[
F(y)+G(y)-\int_{\mathbb{R}}L(\gamma(s),\dot\gamma(s))ds\geq 0.
\]
On the other hand, $F(x)+G(x)-\int_{\mathbb{R}}L(\gamma(s),\dot\gamma(s))ds=0$. Therefore, $F$ is harmonic. The rest follows as in the proof of Theorem \ref{main}.

\end{document}